\documentclass[11pt,oneside,a4paper,mathscr]{amsart}
\usepackage[utf8]{inputenc}

\usepackage{amsthm,amssymb,latexsym,amsmath}
\usepackage[all]{xy}
\usepackage{tikz}

\usepackage[colorlinks,citecolor=blue]{hyperref}

\usepackage{color}
\usepackage[english]{babel}
\usepackage{spectralsequences}
\usepackage{tikz-cd}
\usepackage{tikz}
\usepackage{diagbox}
\usepackage{enumitem}
\usepackage{booktabs}
\usepackage{multirow}
\usepackage{float}

\newcommand{\ch}{\operatorname{ch}}

\newcommand{\Hom}{\operatorname{Hom}}

\DeclareMathOperator{\Coh}{Coh}

\newtheorem{theorem}{Theorem}[section]

\newtheorem{lemma}[theorem]{Lemma}
\newtheorem{corollary}[theorem]{Corollary}

\theoremstyle{definition}
\newtheorem{remark}[theorem]{Remark}

\newtheorem{definition}[theorem]{{\bf Definition}}

\usepackage{yfonts}

\title[A note on Kodaira vanishing]{A note on Kodaira vanishing on surfaces}
\author[Cristian Martinez]{Cristian Martinez}
\address{School of Engineering, Science and Technology\\
Universidad del Rosario \\
Carrera 6 No. 12C-16\\
111711, Bogot\'a\\
Colombia}
\email{cristianm.martinez@urosario.edu.co}


\date{}

\begin{document}

\begin{abstract}
    We give a proof of the Kodaira vanishing theorem on smooth complex surfaces using geometric stability conditions. Likewise, we give a new proof of a result of Xie \cite{XieCounter} characterizing the counterexamples of the Kodaira vanishing theorem in positive characteristic. 
\end{abstract}

\maketitle

\section{Introduction}
A celebrated theorem of Kodaira \cite{KodairaVanishingOriginal} asserts that given an ample line bundle $H$ on a smooth projective variety $X$ over an algebraically closed field of characteristic zero then $H^j(X;H\otimes K_X)=0$ for all $j>0$. The proof of Kodaira's vanishing result uses transcendental methods and there is not an algebraic proof that is characteristic independent. In fact, counterexamples were later constructed by Raynaud \cite{RaynaudCounter} on smooth projective surfaces over algebraic closed fields of characteristic $p>0$. Generalizations of Raynaud's result appear in \cite{MukaiCounter,EKeCounter,Deligne1987Counter,ZHENG2017counter}, and furthermore, a characterization of such counterexamples for the case of surfaces was found in \cite{XieCounter}, where the author proves that a counterexample to Kodaira's result in dimension two can only exist when the surface is of general type or is a quasi-elliptic surface of Kodaira dimension 1.

In this note, we will obtain a similar result using the construction of stability conditions that derive from a recent result of Koseki \cite[Theorem 3.5]{Kos22} about the existence of a Bogomolov-type inequality on smooth projective surfaces over algebraically closed fields of positive characteristic. The precise statement is

\begin{theorem}\label{mainthm}
    Let $X$ be a smooth projective surface over an algebraic closed field $k$. There is a constant $C_{[X]}\geq 0$, depending only on the birational class of $X$ and that vanishes unless $char(k)>0$ and $X$ is either of general type or a quasielliptic surface of Kodaira dimension 1 (see Theorem \ref{KosekiConstant} for the precise definition), such that if $H$ is an ample line bundle with $H^2>6C_{[X]}$ then $H^1(H\otimes K_X)=0$.
\end{theorem}

Our result, in particular, also gives an algebraic proof of Kodaira vanishing in the case of surfaces in characteristic zero. The main technique used in the proof of Theorem \ref{mainthm} is Schur's lemma (Lemma \ref{schur}) and a generalization to positive characteristic of the bounds for the Gieseker chamber obtained in \cite{LoMar23} in the characteristic zero case (Lemma \ref{estimates mini-walls char p}).

\subsection*{Acknowledgments} This note grew as part of the mini-course ``Bridgeland stability conditions and applications'' given by the author at the LEGAL (Libertade em Geometria Alg\'ebrica) meeting in Teres\'opolis, Rio de Janeiro, Brazil, in 2024. The author would like to express his gratitude to the organizers of the LEGAL meeting for the invitation and for a great conference. Also, he would like to thank Omprokash Das, Marcos Jardim, Jason Lo, and Yun Shi for conversations about the topic of this article.

\section{Stability functions and Schur's lemma}
\begin{definition}\label{slopes}
    let $\mathcal{A}$ be an abelian category. An additive function on $\mathcal{A}$ is a homomorphism $f\colon K(\mathcal{A})\rightarrow \mathbb{R}$, where $K(\mathcal{A})$ denotes the Grothendieck group of $\mathcal{A}$. Given a pair of additive functions $d,r\colon K(\mathcal{A})\rightarrow \mathbb{R}$ satisfying the following positivity property:
    \begin{enumerate}
        \item $r(A)\geq 0$ for all $A\in\mathcal{A}$,
        \item if $r(A)=0$ for a nonzero object $A\in \mathcal{A}$ then $d(A)> 0$,
    \end{enumerate}
    we can define the associated slope function
    $$
    \mu_{d,r}(A)=\begin{cases}
        \frac{d(A)}{r(A)} &\text{if}\ r(A)\neq 0\\
        +\infty &\text{otherwise}.
        \end{cases}
    $$
    We say that an object $A\in\mathcal{A}$ is $\mu_{d,r}$-semistable if for every inclusion $A'\hookrightarrow A$ in $\mathcal{A}$ (with $A'$ nonzero) we have $\mu_{d,r}(A')\leq \mu_{d,r}(A)$. We say that $A$ is $\mu_{d,r}$-stable if $A$ is $\mu_{d,r}$-semistable and if the inequality above is always strict unless $A'=A$.
\end{definition}

\begin{definition}
    Let $X$ be a smooth projective variety. A stability condition on $X$ is a pair $\sigma=(Z,\mathcal{A})$, where:
    \begin{enumerate}
        \item $\mathcal{A}\subset D^b(X)$ is the heart of a bounded t-structure;
        \item $Z\colon K(\mathcal{A})\rightarrow \mathbb{C}$ is a homomorphism such that 
        $$
        Z(E)\in \mathbb{H}=\{z\in \mathbb{C}\setminus \{0\}\colon \mbox{arg}(z)\in (0,\pi]\},
        $$
        i.e., $-\mathfrak{Re}(Z)$ and $\mathfrak{Im}(Z)$ satisfy the positivity property of Definition \ref{slopes};
        \item Each non-zero object in $\mathcal{A}$ has a unique Harder-Narasimhan filtration with respect to the slope 
        $$
        \mu_Z(E)=\begin{cases}
        \frac{-\mathfrak{Re}Z(E)}{\mathfrak{Im}Z(E)} &\text{if}\ \mathfrak{Im}Z(E)\neq 0\\
        +\infty &\text{if}\ \mathfrak{Im}Z(E)=0
        \end{cases};
        $$
        \item There exist a finite rank lattice $\Gamma$ and a surjective homomorphism 
        $$
        v\colon K(\mathcal{A})\rightarrow \Gamma
        $$ 
        such that $Z$ factors through $v$, i.e., there exists a homomorphism $\hat{Z}\colon \Gamma\rightarrow \mathbb{C}$ such that $Z=\hat{Z}\circ v$;
        \item There exists a quadratic form $Q$ on $\Gamma\otimes_{\mathbb{Z}} \mathbb{R}$ such that 
        \begin{enumerate}
        \item $Q(v(E))\geq 0$ for all $\mu_Z$-semistable objects $E\in\mathcal{A}$,
        \item $Q(v(E))<0$ for all nonzero object $E\in Ker(Z)$.
        \end{enumerate}
    \end{enumerate}
\end{definition}

Let us denote by $\mbox{Stab}^{\Gamma}(X)$ the set of  stability conditions whose central charge factors through $\Gamma$. This set is indeed a complex manifold \cite{B07}. 

One of the first properties of stable objects is that they satisfy the following vanishing:

\begin{lemma}[Schur's Lemma]\label{schur} Suppose that $\sigma=(Z,\mathcal{A})$ is a stability condition on a smooth projective variety $X$, and that $F,E\in\mathcal{A}$ are $\mu_Z$-stable objects such that $\mu_Z(F)>\mu_Z(E)$. Then $\Hom_{\mathcal{A}}(F,E)=0$. 
\end{lemma}
\begin{proof}
    The proof of this lemma is an easy exercise, but we include it here for completeness. 
    
    Suppose that $\varphi\colon F\rightarrow E$ is a non-zero morphism in $\mathcal{A}$ and let $M=image(\varphi)$. Then we have a surjection $\pi\colon F\twoheadrightarrow M$ and an inclusion $\iota\colon M\hookrightarrow E$, both on $\mathcal{A}$, so that $\varphi=\iota\circ \pi$. Notice that $\mu_Z(E)<\mu_Z(F)\leq +\infty$ by hypothesis. Therefore, since $M\neq 0$ it follows from the stability of $E$ that $\mu_Z(M)<\mu_Z(E)$ unless $M=E$. On the other hand, let $K=ker(\pi)$, then we have the short exact sequence in $\mathcal{A}$
    $$
    0\rightarrow K\rightarrow F\rightarrow M\rightarrow 0,
    $$
    from which we obtain
    $$
    Z(F)=Z(K)+Z(M).
    $$
    Notice that $K\neq F$ because $\varphi$ is non-zero. Also, $K\neq 0$ since otherwise this would imply that $F\cong M$ and so $\mu_Z(F)=\mu_Z(M)\leq \mu_Z(E)$, a contradiction to our hypothesis. Moreover, if $\mu_Z(F)=+\infty$ then the stability of $F$ would force it to be simple and we would have $K=0$, which is impossible. Thus, we have $\mu_Z(K)<\mu_Z(F)<+\infty$, which is equivalent to
    \begin{align*}
        -\mathfrak{Re}(Z(K))\mathfrak{Im}(Z(F))&<-\mathfrak{Re}(Z(F))\mathfrak{Im}(Z(K))\\
         -(\mathfrak{Re}(Z(F))-\mathfrak{Re}(Z(M)))\mathfrak{Im}(Z(F))&<-\mathfrak{Re}(Z(F))(\mathfrak{Im}(Z(F))-\mathfrak{Im}(Z(M)))\\
         \mathfrak{Re}(Z(M))\mathfrak{Im}(Z(F))&<\mathfrak{Re}(Z(F))\mathfrak{Im}(Z(M)),
    \end{align*}
    i.e., $\mu_Z(F)<\mu_Z(M)\leq \mu_Z(E)$, a contradiction.
\end{proof}


\section{A quick review of stability conditions on surfaces in characteristic zero}
Now, let $X$ be a smooth projective surface over an algebraically closed field of characteristic zero, and $\omega, B$ be $\mathbb{Q}$-divisors on $X$ with $\omega$ ample. Then we can define a Bridgeland stability condition $\sigma_{B,\omega}\in \mbox{Stab}^{\Gamma}(X)$, where 
$$
\Gamma=\widetilde{\mathrm{NS}}(X):=\mathbb{Z}\oplus \mathrm{NS}(X)\oplus \frac{1}{2}\mathbb{Z}
$$ 
and $v$ is the Chern character map \cite{B08,AB}. First, let us start by recalling slope stability for sheaves:

 Consider the twisted Mumford slope function  
\[
  \mu_{B,\omega} (E) = \begin{cases} \frac{\omega\ch_1^B(E)}{\omega^2\ch_0(E)} &\text{ if } \ch_0(E) \neq 0, \\
  +\infty &\text{ otherwise,}
  \end{cases}
\]
where $\ch^B(E)=\exp(-B)\ch(E)$. A nonzero coherent sheaf $E$ on $X$ is then called $\mu_{B,\omega}$-semistable (resp.\ $\mu_{B,\omega}$-stable) if
\[
  \mu_{B,\omega}(F) \leq (\text{resp.\ $<$})\,\, \mu_{B,\omega}(E)
\]
for every nonzero proper subsheaf $F\hookrightarrow E$. In particular, a $\mu_{B,\omega}$-semistable sheaf of nonzero rank is torsion-free. In such case, we have $\mu_{B,\omega}(E)=\mu_{0,\omega} (E)-B\omega$, meaning $\mu_{B,\omega}$-semistability is equivalent to $\mu_{0,\omega}$-stability for $E$ (the usual Mumford stability).

An important property of $\mu_{B,\omega}$-semistable sheaves is that they satisfy the Bogomolov inequality 
\begin{equation}\label{Bogomolovineq}
\Delta(E):=\ch_1(E)^2-2\ch_0(E)\ch_2(E)\geq 0.
\end{equation}
Second, we define the following full subcategories of $\Coh (X)$:
\begin{align*}
  \mathcal{T}_{B,\omega} &= \{ E \in \Coh (X)\colon \mu_{B,\omega}(Q)>0\ \text{for all sheaf quotients}\ E\twoheadrightarrow Q \}, \\
   \mathcal{F}_{B,\omega} &= \{ E \in \Coh (X) \colon \mu_{B,\omega}(A)\leq 0\ \text{for all subsheaves}\ A\hookrightarrow E\}.
\end{align*}
These subcategories form a torsion pair and so the extension closure in $D^b(X)$
\[
  \Coh^{B,\omega}(X) = \langle \mathcal{F}_{B,\omega}[1], \mathcal{T}_{B,\omega} \rangle
\]
is the heart of a t-structure. The function
\[
Z_{B,\omega,t}(E) = -\ch_2^B(E) + \frac{t^2\omega^2}{2}\ch_0^B(E) + i \omega\ch_1^B(E)
\]
is the central charge of a stability condition $\sigma_{B,\omega,t }$ on $\Coh^{B,\omega}(X)$ for every $t>0$ (see \cite{AB} and also  \cite{B08} for the positivity and HN properties, and \cite[Theorem 3.23]{Toda2012StabilityCA} for the support property). The corresponding Bridgeland slope is denoted by $\nu_{B,\omega,t}$, i.e.,
$$
\nu_{B,\omega,t}(E)=\begin{cases}
\frac{\ch_2^B(E)-\frac{t^2\omega^2}{2}\ch_0(E)}{\omega\ch_1^B(E)} &\text{if}\ \ \omega\ch_1^B(E)\neq 0\\
+\infty &\text{otherwise.}
\end{cases}
$$

\begin{definition}
Let $B$ be a $\mathbb{Q}$-divisor. A torsion-free sheaf $E$ is called $B$-twisted Gieseker semistable with respect to an ample class $\omega$ if and only if
\begin{itemize}
    \item[(i)] $E$ is $\mu_{B,\omega}$-semistable, and
    \item[(ii)] For every subsheaf $A\hookrightarrow E$ such that $\mu_{B,\omega}(A)=\mu_{B,\omega}(E)$ we have
    $$
    \frac{\chi(A\otimes L )}{\ch_0(A)} \leq \frac{\chi(E\otimes L)}{\ch_0(E)},
    $$
    where $L=-B+\frac{K_X}{2}$ and $\chi(\_\otimes L)$ is computed using the Hirzebruch-Riemann-Roch formula.
\end{itemize}
\end{definition}

\begin{remark}\label{largevolumelimit} As proven by Bridgeland \cite[Prop.\  14.2]{B08} and Lo-Qin \cite[Theorems 1.1, 1.2(i)]{LQ}, the only objects with non-negative rank and positive degree (with respect to some polarization $\omega$ and $\mathbb{R}$-divisor $B$) that remain $\nu_{B,\omega,t}$-semistable for large values of the parameter $t$ are precisely the $B$-twisted Gieseker semistable sheaves with respect to $\omega$. 
\end{remark}

\begin{remark}\label{rem:BGineq}
Given $\mathbb{Q}$-divisors $B, \omega$ on $X$, the numerical discriminant $\Delta_{B,\omega}$ of an object $E \in D^b(E)$ is
\[
  \Delta_{B,\omega}(E) := (\omega \ch_1^B(E))^2 - 2\omega^2\ch_0^B(E)\ch_2^B(E).
\]
When $\omega$ is ample, every $\nu_{B,\omega,t }$-semistable object $E$ satisfies the Bogomolov-Gieseker inequality
\[
  \Delta_{B,\omega}(E) \geq 0
\]
(see \cite[Theorem 6.13]{MS}). 
\end{remark}
The following lemmas are stated here for completeness of the presentation.
\begin{lemma}\label{equalslope}\cite[Lemma 4.6]{LoMar23}
If $E$ is a $B$-twisted $\omega$-semistable sheaf of positive rank and $A\hookrightarrow E$ is a subobject of $E$ in $\Coh^{B,\omega}(X)$ such that $\mu_{B,\omega}(A)=\mu_{B,\omega}(E)$ and $\nu_{B,\omega,t}(A)=\nu_{B,\omega,t}(E)$, then $A$ is a subsheaf of $E$ that makes $E$ a properly $B$-twisted $\omega$-semistable sheaf. In particular, $A$ can never destabilize $E$ for any value of $t$.  
\end{lemma}


\begin{lemma}\label{estimates a-mini-walls}\cite[Lemma 4.8]{LoMar23}
Let $B,\omega\in \mathrm{NS}(X)_{\mathbb{Q}}$ with $\omega$ ample and integral. Let $E$ be a $B$-twisted $\omega$-Gieseker semistable sheaf with $\ch_0(E)>0$ and $\omega\ch_1^B(E)> 0$.  Then $E$ is $\nu_{B,\omega,t}$-semistable for 
$$
t^2\omega^2> \mu_{B,\omega}(E)\Delta_{B,\omega}(E).
$$
\end{lemma}

\begin{lemma}\label{duality}\cite[Lemma 4.2]{martinez2017duality}
    Suppose that $E\in \Coh^{B,\omega}(X)$ is such that $\omega\ch^B_1(E)>0$, then $E$ is $\nu_{B,\omega,t}$-(semi)stable if and only if $E^{\vee}[1]$ is $\nu_{-B,\omega,t}$-(semi)stable. 
\end{lemma}

\section{Kodaira-type vanishing for surfaces over arbitrary algebraically closed fields}

Let us start by recalling a result by Koseki (also known as the Bogomolov-Gieseker-Koseki inequality):

\begin{theorem}\label{KosekiConstant}\cite[Theorem 3.5]{Kos22} Let $X$ be a smooth projective surface over an algebraically closed field of positive characteristic. Then there exists a constant $C_{[X]}\geq 0$ (depending only on the birational class of $X$) such that for every numerically nontrivial nef divisor $H$ and every $\mu_H$-semistable torsion-free sheaf $E\in \Coh(X)$, we have
$$
\tilde{\Delta}(E):=\ch_1(E)^2-2\ch_2(E)\ch_0(E)+C_{[X]}\ch_0(E)^2\geq 0.
$$
Moreover, the constant $C_{[X]}$ can be taken as follows:
\begin{enumerate}
    \item If $X$ is a minimal surface of general type, then 
    $$
    C_{[X]}=2+5K_X^2-\chi(\mathcal{O}_X).
    $$
    \item If $\kappa(X)=1$ and $X$ is quasi-elliptic, then 
    $$
    C_{[X]}=2-\chi(\mathcal{O}_X).
    $$
    \item Otherwise, $C_{[X]}=0$.
\end{enumerate}
\end{theorem}

\begin{corollary}
    Let $E\in \Coh(X)$ be a torsion-free $\mu_H$-semistable sheaf such that $\ch_1^{B}(E)\cdot H=0$, then
    $$
    \ch_2^B(E)-\left(\frac{C_{[X]}}{2H^2}+\frac{\alpha^2}{2}\right)\ch_0(E)H^2<0
    $$
    for all $\alpha>0$. As a consequence, 
    $$
    Z_{B,H,\alpha}^{C_{[X]}}:=-\ch_2^B+\left(\frac{C_{[X]}}{2H^2}+\frac{\alpha^2}{2}\right)\ch_0H^2+\sqrt{-1}\ch_1^B\cdot H
    $$
    is the central charge of a stability condition on $\Coh^{B,H}(X)$, which satisfies the support property with respect to the quadratic form 
    $$
    Q=\tilde{\Delta}+C_H(\ch_1^B\cdot H)^2,
    $$
    where $C_H\geq 0$ is a constant (only depending on $H$) such that
    $$
    C_H(HD)^2+D^2\geq 0
    $$
    for all effective divisors $D$ on $X$. We denote the corresponding central charge by $\nu^{C_{[X]}}_{B,H,\alpha}$.
\end{corollary}
\begin{remark}
    Notice that in the case when $C_{[X]}=0$ we have $\widetilde{\Delta}=\Delta$,  and $Z^{C_{[X]}}_{B,H,\alpha}=Z_{B,H,\alpha}$. Moreover, the quadratic form becomes $Q=\Delta+C_{H}(\ch_1^B\cdot H)^2$. Therefore, for numerical computations we could assume that the characteristic zero case is included in the case when $C_{[X]}=0$.
\end{remark}
\begin{remark}\label{reparametrization}
    Notice that under the re-parametrization 
    $$
    t^2=\frac{C_{[X]}}{H^2}+\alpha^2,
    $$
    the stability condition $(Z_{B,H,\alpha}^{C_{[X]}},\Coh^{B,H}(X))$ coincides with $(Z_{B,H,t},\Coh^{B,H}(X))$ (even though they satisfy the support property with respect to different quadratic forms). In particular, every $B$-twisted $H$-Gieseker semistable sheaf $E$ (similarly $E[1]$) is $\nu^{C_{[X]}}_{B,H,\alpha}$-semistable for $\alpha\gg 0$, whenever $E\in \Coh^{B,H}(X)$ (respectively, $E[1]\in \Coh^{B,H}(X)$). However, the bound of Lemma \ref{estimates a-mini-walls} does not work because precisely the support property with respect to $\Delta_{B,H}$ fails in positive characteristic. 
\end{remark}
\begin{lemma}\label{numcharpDiscriminant}
Let $X$ be a smooth projective surface over an algebraically closed field of positive characteristic and let $C_{[X]}$ as defined in Theorem \ref{KosekiConstant}. Then, any $\nu_{B,H,\alpha}^{C_{[X]}}$-semistable object $E\in\Coh^{B,H}(X)$ satisfies
\begin{equation}\label{charpDiscriminant}
\widetilde{\Delta}_{B,H}(E):=(H\ch_1^B(E))^2-2H^2\ch_0^B(E)\ch_2^B(E)+C_{[X]}\ch_0(E)^2\geq 0.
\end{equation}
\end{lemma}
\begin{proof}
This is exactly as the proof of \cite[Theorem 6.13]{MSlec} since by the Hodge Index Theorem and Theorem \ref{KosekiConstant} we have that $\widetilde{\Delta}_{B,H}(E)\geq 0$ for all $\mu_H$-semistable sheaves, $\widetilde{\Delta}_{B,H}$ is negative definite on the kernel of $Z_{B,H,\alpha}^{C_{[X]}}$ and, as mentioned in Remark \ref{reparametrization}, $\nu_{B,H,\alpha}^{C_{[X]}}$-stability for large values of $\alpha$ coincides with $\nu_{B,H,t}$-stability. 
\end{proof}

\begin{lemma}\label{estimates mini-walls char p}
Let $B,H\in \mathrm{NS}(X)_{\mathbb{Q}}$ with $H$ ample and integral. Let $E$ be a $B$-twisted $H$-Gieseker semistable sheaf with $\ch_0(E)>0$ and $H\ch_1^B(E)> 0$.  Then $E$ is $\nu_{B,\omega,\alpha}^{C_{[X]}}$-semistable for 
$$
\alpha^2H^2> \mu_{B,H}(E)\widetilde{\Delta}_{B,H}(E).
$$
\end{lemma}
\begin{proof}
    This follows the same steps as the proof of Lemma \ref{estimates a-mini-walls} (\cite[Lemma 4.8]{LoMar23}). Moreover,  Lemma \ref{estimates a-mini-walls} can be recovered from this proof setting $C_{[X]}=0$. 
    
    Since every $B$-twisted $H$-Gieseker semistable sheaf in $\Coh^{B,H}(X)$ is $\nu^{C_{[X]}}_{B,H,\alpha}$-semistable for $\alpha\gg 0$ (see Remark \ref{largevolumelimit} and Remark \ref{reparametrization}), the idea will be to show that, when varying $\alpha$, $E$ can only be destabilized for values of $\alpha$ smaller than the bound above.
    
    First, suppose that $E$ is $\mu_{B,H}$-semistable with $\mu_{B,H}(E)>0$, then the long exact sequence of cohomology for a destabilizing sequence $0\rightarrow A\rightarrow E\rightarrow Q\rightarrow 0$ in $\Coh^{B,H}(X)$ gives
$$
0\rightarrow \mathcal{H}^{-1}(Q)\rightarrow A\rightarrow E\rightarrow \mathcal{H}^0(Q)\rightarrow 0.
$$
   Since $E$ is a torsion-free sheaf, this shows that $A$ is a sheaf of positive rank and, moreover, $\mu_{B,H}(A)\leq \mu_{B,H}(E)$. Indeed, if $\mu_{B,H}(A)>\mu_{B,H}(E)$ and $A_1\hookrightarrow A$ is the first Harder-Narasimhan factor of $A$, then $\mu_{B,H}(A_1)>\mu_{B,H}(E)>0$ and so $A_1$ must be a subsheaf of $\mathcal{H}^{-1}(Q)$. This is a contradiction since $\mathcal{H}^{-1}(Q)\in \mathcal{F}_{B,H}$ and so $\mu_{B,H}$ is non-positive for each of its subobjets.
    
    Now, $H\ch_1^B(Q)>0$ since otherwise $A$ would not destabilize $E$. Combining this with Lemma \ref{equalslope}, we have shown 
    $$
    0<H\ch_1^B(A)<H\ch_1^B(E),\ \ \ 0<\mu_{B,H}(A)<\mu_{B,H}(E).
    $$
    Moreover, the wall produced by $A$ has equation
    $$
    \left(C_{[X]}+\alpha^2H^2\right)(\mu_{B,H}(E)-\mu_{B,H}(A))=\frac{2\ch_2^B(A)\mu_{B,H}(E)}{\ch_0(A)}-\frac{2\ch_2^B(E)\mu_{B,H}(A)}{\ch_0(E)}.
    $$
    Since $A$ is $\nu_{B,H,\alpha}^{C_{[X]}}$-semistable on such wall, Lemma \ref{numcharpDiscriminant} implies that
    $$
    \frac{2\ch_2^B(A)}{\ch_0(A)}\leq H^2\mu_{B,H}(A)^2+\frac{C_{[X]}}{H^2}.
    $$
    Thus,
    \footnotesize{
    \begin{align*}
        \left(C_{[X]}+\alpha^2H^2\right)(\mu_{B,H}(E)-\mu_{B,H}(A))&\leq \mu_{B,H}(A)\left(H^2\mu_{B,H}(A)\mu_{B,H}(E)-\frac{2\ch_2^B(E)}{\ch_0(E)}\right)+\frac{C_{[X]}\mu_{B,H}(E)}{H^2} \\
        &< \mu_{B,H}(A)\left(\frac{(H\ch_1^B(E))^2}{H^2\ch_0(E)^2}-\frac{2\ch_2^B(E)}{\ch_0(E)}\right)+\frac{C_{[X]}\mu_{B,H}(E)}{H^2} \\
        &= \mu_{B,H}(A)\left(\frac{\Delta_{B,H}(E)}{H^2\ch_0(E)^2}\right)+\frac{C_{[X]}\mu_{B,H}(E)}{H^2} \\
        &= \mu_{B,H}(A)\left(\frac{\widetilde{\Delta}_{B,H}(E)}{H^2\ch_0(E)^2}\right)+\frac{C_{[X]}}{H^2}(\mu_{B,H}(E)-\mu_{B,H}(A))\\
        &\leq \mu_{B,H}(A)\left(\frac{\widetilde{\Delta}_{B,H}(E)}{H^2\ch_0(E)^2}\right)+C_{[X]}(\mu_{B,H}(E)-\mu_{B,H}(A)).
    \end{align*}
    }
    \normalsize
    Therefore,
    \begin{align*}
        \alpha^2H^2(\mu_{B,H}(E)-\mu_{B,H}(A))&< \frac{\mu_{B,H}(E)\widetilde{\Delta}_{B,H}(E)}{H^2\ch_0(A)\ch_0(E)}.
    \end{align*}
    Now, since $H$ is integral and 
    $$
    \mu_{B,H}(E)-\mu_{B,H}(A)=\mu_{0,H}(E)-\mu_{0,H}(A)>0
    $$
    we conclude that
    $$
    \alpha^2H^2<\mu_{B,H}(E)\widetilde{\Delta}_{B,H}(E).
    $$
\end{proof}

One of the advantages of having several ways of measuring stability is that through Schur's lemma we can obtain the vanishing of certain cohomology groups. For instance, we are now ready for the proof of our main theorem:


\begin{proof}[Proof of Theorem \ref{mainthm}]
This is a direct application of Lemma \ref{schur}. The idea is to find a stability condition $\sigma=(Z,\mathcal{A})\in\mbox{Stab}(X)$ such that: 
    \begin{itemize}
    \item[(1)] $\displaystyle H,\ \mathcal{O}[1]\in \mathcal{A}$ are $\sigma$-stable,
    \item[(2)] $\displaystyle \mu_{Z}(H)>\mu_{Z}(\mathcal{O}[1])$.
    \end{itemize}
    Thus, we will have 
    $$
    H^1(H\otimes K_X)=\mbox{Ext}^1(\mathcal{O},H\otimes K_X)\cong\mbox{Ext}^1(H,\mathcal{O})^*=\mbox{Hom}(H,\mathcal{O}[1])^*=0.
    $$
    Then, it will be enough to find a values of $\alpha>0$ and $\beta$ such that for $B=\beta H$
    \begin{enumerate}
        \item $0<\beta<1$,\ \  so that $H,\mathcal{O}[1]\in \Coh^{B,H}(X)$;
        \item $H$ and $\mathcal{O}[1]$ are $\nu_{B,H,\alpha}^{C_{[X]}}$-stable;
        \item and $\displaystyle \frac{C_{[X]}}{H^2}+\alpha^2<\beta(1-\beta)$,\ \  so that $\nu_{B,H,\alpha}^{C_{[X]}}(H)>\nu_{B,H,\alpha}^{C_{[X]}}(\mathcal{O}[1])$.
        \end{enumerate}
   To get condition (2) we use Lemma \ref{estimates mini-walls char p} and Lemma \ref{duality} and obtain 
   \begin{itemize}
       \item[(i)] $H$ is  $\nu_{B,H,\alpha}^{C_{[X]}}$-stable if 
       $$
       \alpha^2H^2>(1-\beta)C_{[X]};
       $$
       \item[(ii)] $\mathcal{O}[1]$ is  $\nu_{B,H,\alpha}^{C_{[X]}}$-stable if 
       $$
       \alpha^2H^2>\beta C_{[X]}.
       $$
   \end{itemize}
   Taking $\beta=1/2$ we see that conditions (i) and (ii) are the same and so it is enough to show that there exists $\alpha>0$ such that
   $$
   \left(1-\frac{1}{2}\right)C_{[X]}<\alpha^2H^2<\frac{1}{2}\left(1-\frac{1}{2}\right)H^2-C.
   $$
   Now, by our assumption that $H^2>6C_{[X]}$ we know that
   $$
       \frac{1}{4}H^2-C>\frac{3}{2} C_{[X]}-C_{[X]}= \frac{1}{2}C_{[X]}.
   $$
   Thus, it is always possible to find the required $\alpha$.
\end{proof}

\begin{remark}
    Theorem 2.15 agrees with the results in \cite{XieCounter}, stating that counterexamples for the Kodaira Vanishing Theorem can only exist when $C_{[X]}>0$. See also \cite{Enokizono2021}, \cite{DiCerbo15}, and \cite{ZHANG23}.
\end{remark}

\bibliographystyle{hep}
\bibliography{refs.bib}

\end{document}